\documentclass[11pt,english]{smfart}
\usepackage[cmtip,arrow]{xy}
\usepackage{pb-diagram,pb-xy}
\usepackage{epsfig}
\usepackage{amsmath}
\usepackage{amssymb}
\usepackage{amscd}
\usepackage{latexsym}
\usepackage{tabularx}
\usepackage{a4wide}
\usepackage{color}
\usepackage{enumerate}
\usepackage{subfigure}
\usepackage{url}
\usepackage{graphics}
\usepackage{mathrsfs}

\newtheorem{theorem}{Theorem}
\newtheorem{definition-theorem}[theorem]{Definition-Theorem}

\newtheorem{proposition}[theorem]{Proposition}
\newtheorem{lemma}[theorem]{Lemma}

\newtheorem{remark}[theorem]{Remark}

\newcommand{\out}{\mathrm{out}}

\renewcommand{\H}{\mathrm{F}}
\renewcommand{\int}{\mathrm{int}}
\newcommand{\Vor}{\mathrm{Vor}}
\newcommand{\FP}{\mathcal{FP}}
\newcommand{\SC}{\mathcal {SC}}
\newcommand{\E}{\mathbb E}

\newcommand{\supp}{\mathrm{supp}}

\newcommand{\Cay}{\mathrm{Cay}}
\newcommand{\codim}{\mathrm{codim}}
\newcommand{\Eu}{\Xi_1}
\newcommand{\conv}{\textrm{conv-hull}}

\newcommand{\cov}{\mathrm{Cov}}
\newcommand{\din}{d^{\mathrm{\,in}}}
\newcommand{\dout}{d^{\mathrm{\,out}}}

\title{Lattice of Integer Flows  and  Poset of Strongly Connected Orientations}
\author{Omid Amini}
\address{CNRS-CMLS, \'Ecole Polytechnique, Palaiseau, France}
\email{omid.amini@polytechnique.edu}

\thanks{Special thanks to Madhusudan Manjunath for helpful discussions and Mathieu Desbrun for his support. 
\\
This work was done in 2010 while the author was visiting the California Institute of Technology. }
\begin{document}
\begin{abstract}
We show that the Voronoi cells of the lattice of integer flows of a finite connected graph $G$ in the quadratic vector space of real valued flows have the following very precise combinatorics: the face poset of a Voronoi cell is isomorphic to the poset of strongly connected orientations of subgraphs of $G$. This confirms a recent conjecture of Caporaso and Viviani $\{$\textsl{Torelli Theorem For Graphs and Tropical Curves}, Duke Math. J. \textbf{153}(1) (2010), 129-171$\}$.
\end{abstract}

\maketitle

\section{Introduction}\label{sec:intro}

\subsection{Lattice of integer flows}
Consider a finite graph $G = (V,E)$ with possible parallel edges and loops. 
All through the paper, we suppose that $G$ is connected. Following the classical works (e.g.,\cite{Ser77}, \cite{BHN97}), we replace the set of edges of $G$ with the set of oriented edges $\mathbb E$, where each edge $e\in E$ is replaces with two oriented edges in opposite directions in $\mathbb E$. By the abuse of notation, we also denote by $e$ an element of $\mathbb E$, and use $\bar e$ when referring to the same edge but with the inverse orientation.

For a coefficient ring $A$, let $C_0(G,A)$ be the free $A$-module generated by the set of vertices $V$ and let $C_1(G,A)$ be the quotient of the free $A$-module generated by the set of oriented edges $\mathbb E$ by all the relations $(e)=-(\bar e)$ for any oriented edge $e\in \mathbb E$. The basis elements of $C_0(G,A)$ are denoted by $(v)$ for $v\in V$ and the basis elements of $C_1(G,A)$ are denoted by $(e)$ for $e\in \mathbb E$. The chain complex $\mathcal C_*(G,A): C_1(G,A)\stackrel{\partial}{\longrightarrow} C_0(G,A)$ is the usual simplicial chain complex of $G$: the boundary map is given by $\partial (e) = (u)-(v)$ for $u$ and $v$ being respectively the head and the tail of the oriented edge $e$. 

Similarly, one can define the cochain complex $\mathcal C^*: C^0(G,A)\stackrel {d}{\longrightarrow} C^{1}(G,A)$ by letting $C^0(G,A)$ to be the $A$-module of all the functions $f: V \rightarrow A$ and  $C^1(G,A)$ to be the $A$-module of all the functions $g: \mathbb E \rightarrow A$ verifying the property that $g(e) = -\, g(\bar e)$ for any $e\in \E$. The differential $d$ is defined as follows: $d(f)(e) = f(u)-f(v)$ where $u$ and $v$ are respectively the head and the tail of the oriented edge $e$. Note that the spaces $C_i(G,A)$ and $C^i(G,A)$ are canonically dual for $i=0,1$. 

 The (co)homology of  $\mathcal C_*$ and $\mathcal C^*$ have a rather simple description: 
 \begin{itemize}
 \item $H_1(\mathcal C_*(G,A))\simeq A^g$, i.e.,  is a free $A$-module of rank $g$ where $g$ is the genus of the (non-oriented) graph $G$. $H_0(C_*(G,A)) \simeq A$. 
 \item Similarly, $H^1(\mathcal C_*(G,A)) \simeq A^g$ and $H^0(\mathcal C^*(G,A))\simeq A$.
 \end{itemize}

The first homology group $H_1(\mathcal C_*(G,A))$ also coincides with the $A$-module of all the $A$-valued flows in $G$. 
Recall that a \emph{flow} $x$ in $G$ with values in $A$ is the data of $\{x_e \in A\}_{e\in \mathbb E}$ for any $e\in \mathbb E$ with the following properties:
\begin{itemize}
\item For any oriented edge $e\in \mathbb E$, $x_e = -x_{\bar e}$.
\item For any vertex $v$, let $\mathbb E^{\out}(v)$  be the set of all the oriented edges emanating from $v$, i.e., with tail equal to $v$. Then $\sum_{e \in \mathbb E^{\out}(v)}\,x_{e} = 0$. (In other words: the amount of the flow coming out of $u$ is equal to the amount of the flow entering $u$.) 
\end{itemize}

Let $\H$ be the vector space of  real-valued flows in $G$, i.e., $\H:=H_1(\mathcal C_*(G,\mathbb R))$. And let $\Lambda$ be the lattice of integer-valued flows in $\H$, i.e., $\Lambda:=H_1(\mathcal C_*(G,\mathbb Z))$, $\Lambda \subset \H$. Note that $\H$ is a vector subspace of dimension $g$ in $C_1(G,\mathbb R)$ and $\Lambda$ has rank $g$.

\subsection{Positive quadratic form $q$} There is a natural inner-product on  $C_1(G,\mathbb R)$ (and a natural inner product on $C_0(G,\mathbb R)$) given by the simplicial structure of $G$. This is given on the basis as follows

\begin{align*}
\langle (e),(e') \rangle :=
\left\{
\begin{array}{rl}
1  & \mbox{if } e' = e, \,\,\, \textrm{and}\\
-1 &  \hspace{.9cm} \mbox{if } e' = \bar e;\\
0 & \mbox{otherwise.}
\end{array}
\right.
\end{align*}

(Similarly on $C_0(G,\mathbb R)$, $\langle (u),(v) \rangle := \delta_{u,v}$ for vertices $u,v \in V$.)

\vspace{.2cm}
\noindent It is straightforward to check that the above pairings naturally identify the dual spaces $C_i(G,\mathbb R)$ and $C^i(G,\mathbb R)$ and the adjoint $d^*$ of $d$ gets identified with $\partial$.
 
 \vspace{.4cm}

The above inner-product defines by restriction an inner-product on $\H$, denoted again by $\langle\,,\rangle$;  let $q$ be the corresponding positive quadratic form on $F$ given by $q(x) = \langle x,x \rangle$, for $x\in \H$.

\subsection{Voronoi diagram of $\Lambda$ in $H$} Consider a discrete subset $S$ in the real vector space $\H$. For a point $\lambda$ in $S$, we define the {\it Voronoi cell} of $\lambda$ with respect to $q$  as 
$$V_\lambda\:=\:\Bigl\{x \in \H \:\:|\: \:q(x-\lambda)\leq q(x-\mu)\:\: \textrm{for any other point}\:\: \mu \in S\:\Bigr\}.$$

\noindent The {\it Voronoi diagram} $\Vor_q(\mathcal S)$  is the decomposition of $\H$ induced by the cells $V_\lambda$, for $\lambda\in\mathcal S\:.$  It is straightforward to see that each cell of the Voronoi diagram is a polytope  or a polyhedron (in the case there are cells of infinite volume).

Consider the discrete subset $\Lambda$ of $\H$, which is of full rank.  By translation-invariance of the distances defined by $q$, the cells of the Voronoi diagram are all simply translations of each other, i.e., for a point $\lambda $ in $\Lambda$, $V_\lambda = V_O+\lambda$, where $O$ denotes the origin.  Thus, the Voronoi cell decomposition of $\H$ is completely understood by the Voronoi cell $V_{O}$, which is easily seen to be a polytope.

\vspace{.4cm}

 Consider now the set of all faces of the polytope $V_0$ ordered by inclusion. They form a finite poset that we denote by $\FP$. Our aim in this paper is to give a precise description of this poset in terms of the original graph $G$, that we now explain.

\subsection{Poset of strongly connected orientations of subgraphs} Let $D$ be an orientation of a connected graph $G$. $D$ is called {\it strongly connected} if any pair of vertices $u$ and $v$ in $D$ are connected by an oriented path from $u$ to $v$ and an oriented path from $v$ to $u$. An orientation of a (non necessarily connected) graph $G$ is called strongly connected if the orientation induced on each of the connected components is strongly connected.  

 Let $G$ be a given graph (with possibly multiple edges and loops). Define the following poset $\SC$, that we call the poset of strongly connected orientations of subgraphs of $G$: the elements of $\SC(G)$ are all the pairs $(H,D_H)$ where $H$ is a subgraph of $G$ and $D_H$ is a strongly connected orientation of the edges of $H$. (Note that the vertices and the edges of $G$ are labeled, so parallel edges are distinguished in dealing with subgraphs.) A partial order $\preceq$ is defined on $\SC$ as follows: given two elements $(H,D_H)$ and $(H',D_{H'})$ in $\SC$, we have $(H,D_H) \preceq (H',D_{H'})$ if and only if $H' \subseteq H$ and the orientation $D_{H'}$ is the orientation induced by $D_H$ on $H'$. It is quite straightforward to see that $(\SC,\preceq)$ is a graded poset, and the grading is given by $g$ minus the genus of the underlying subgraph. Note that $(\emptyset,\emptyset)$ is the maximum element of $(\SC,\preceq)$.

\vspace{.4cm}

 The main result of this paper is the following theorem.
 
 \begin{theorem}\label{thm:main}
 The two posets $\FP$ and $\SC$ are isomorphic. 
 \end{theorem}

 This is precisely Conjecture 5.2.8 of Caporaso and Viviani~\cite{CV10}, which answers as well a question asked by Bacher, de la Harpe, and Nagnibeda in~\cite{BHN97}. The two posets $\FP$ and $\SC$ are the posets $\mathrm{Faces}(\Vor_\Gamma)$ and $\mathcal{OP}_\Gamma$ in the notations of~\cite{CV10} where $G=\Gamma$ is the metric graph with all lengths equal to one. The proof of our theorem can be extended  in a straightforward way to more general metric graph, thus, Conjecture 5.2.8 of~\cite{CV10} holds.  We have presented the proofs in the unweighted case to simplify the presentation.  Note that Conjecture 5.2.8 of~\cite{CV10} has a second part (explained in Section~\ref{sec:secondpart}): we will establish a precise bijection between the elements of the two posets, from this bijection the second part of the conjecture easily follows.

\section{Proof of Theorem~\ref{thm:main}}

A (non-necessarily connected) graph is Eulerian if every vertex has an even degree. An orientation of an Eulerian graph is called Eulerian if each vertex has the same in- and out-degrees. It is clear that every Eulerian orientation of a graph is necessarily strongly connected (in the sense we describes above). Also note that an Eulerian orientation $D$  of a graph $H$ has a canonical flow denoted by $x^D$. This is defined by setting $x^D_e=1$ if $e \in D$ (thus, $x_e^D = -1$ if $\bar e\in D$) and $x^D_e =0$ if neither $e$ nor $\bar e$ is in $D$. For simplicity, we say a flow $x$ in a graph $G$ is Eulerian if it is of the form $x^D$ for some $D$ an Eulerian orientation of a subgraph $H$ of $G$. The genus of an Eulerian flow $\mu=x^D$, for $D$ an Eulerian orientation of an Eulerian graph $H$, is by definition the genus of  $H$. In the case $H$ is disconnected, this latter is the sum of the genus of the connected components of $H$. 

Given a flow $x$, by the support of $x$, denoted by $\supp(x)$, we mean the set of all the oriented edges $e\in \mathbb E$ with $x_e >0$. Clearly the oriented graph defined by $\supp(x)$ is strongly connected. (Different flows can however have the same support.) Note also that for an Eulerian flow $x$, $\supp(x)$ is the Eulerian orientation with $x= x^{\supp(x)}$.

\begin{lemma}[Intersecting Voronoi Cells]\label{lem:eulerian}
A Voronoi cell $V_\lambda$, for $\lambda \in \Lambda$,  intersects the Voronoi cell $V_O$ of the origin if and only if $\lambda$ is Eulerian. For two points $\lambda$ and $\mu$ in $\Lambda$, $V_\lambda \cap V_\mu \neq \emptyset$ if and only if $\lambda - \mu$ is Eulerian. 
\end{lemma}

\begin{proof} We first show that if $V_\lambda \cap V_O \neq \emptyset$, then $\lambda/2 \in V_\lambda \cap V_O $. This will then easily implies the lemma.
  
\noindent The claim is a particular property of the Voronoi diagrams defined by the lattices.  Suppose that $\lambda/2 \notin V_\lambda \cap V_O$. This means there exists a point $\mu \in \Lambda$ such that $q(\mu-\lambda/2) < q(\lambda/2)$, in other words, $\langle\lambda,\mu\rangle > q(\mu)$ or equivalently, $\langle \lambda- \mu, \mu \rangle > 0$. Consider now the parallelogram defined by $O$, $\lambda$, $\mu$ and $\lambda-\mu$ (which is in $\Lambda$). We show that a point $x \in V_O\cap V_\lambda$ is certainly closer in distance to one of $\mu$ or $\lambda-\mu$, and this will be a contradiction. 
By assumption $q(x) = q(x-\lambda)$, which gives $2\langle x,\lambda \rangle = q(\lambda)$. Since $q(x) \leq q(x-\mu)$, we have $2\langle x,\mu\rangle \leq q(\mu)$. Combining the two (in-)equalities, we have 
$$2 \langle x , \lambda -\mu \rangle \geq q(\lambda) - q(\mu) = q(\lambda -\mu) + 2\langle \lambda -\mu, \mu \rangle > q(\lambda -\mu).$$
This shows that  $q(x) > q(x-\lambda+\mu)$, which is a contradiction to the assumption $x \in V_O$. And the claim follows. 

Suppose now $\lambda \in \Lambda$ a lattice point such that $\lambda/2 \in V_O \cap V_\lambda$. This latter condition is equivalent to the following:
\begin{center}
For all $\mu \in \Lambda$, \qquad $\langle \lambda, \mu \rangle \leq q(\mu). $
\end{center}
Let $D = \supp(\lambda)$. We observed that $D$ is a strongly connected orientation. For any oriented cycle $C$ in $D$, consider the flow $x^C$ in $G$. We must have $\langle \lambda, x^C \rangle = \sum_{e\in C} \lambda_e \geq q(C).$ By combining these two inequalities, we infer that the equality holds, and so, $\lambda_e = 1$ for all $e\in C$. Since every oriented edge in $\supp(\lambda)$ is contained in an oriented cycle $C \subset \supp(\lambda)$, we have $\mu_e =1$ for all $e\in \supp(\lambda)$.  This shows that $\lambda$ is Eulerian and the first claim of the lemma follows.  The second claim follows by translation invariance. 
\end{proof}

\subsection{Codimension one faces of $\FP$: Description of the Delaunay Edges}
We provide now a description of the faces of codimension one in $\FP$; this clearly provides a description of the dual Delaunay edges.  (Recall that the one-skeleton of the Delaunay dual is the graph obtained on $\Lambda$ by joining two points $\lambda$ and $\mu$ if $V_\lambda \cap V_\mu$ is a face of codimension one in both $V_\lambda$ and $V_\mu$.)

Let $F \in \FP$ be a face of codimension one and let $x$ denote a generic point of $F$ (an arbitrary point in the relative interior of $F$).  Since $F$ is of codimension one, there is a Voronoi cell $V_\lambda$ such that $F \subseteq V_O \cap V_\lambda$. By Lemma~\ref{lem:eulerian}, $\lambda$ is Eulerian. 

\begin{lemma}\label{lem:codim}
The intersection $V_O \cap V_\lambda$ lies in an affine plane of codimension equal to the genus of $\supp(\mu)$.  In particular, if $V_O$ and $V_\lambda$ intersect in a face of codimension one, then $\supp(\mu)$ is  a circuit  $C$ (an orientation of a cycle) and $\mu = x^{C}$.
\end{lemma}

\begin{proof} A point $x$ in the intersection $V_O\cap V_\lambda$  has the following properties:
\begin{itemize}
\item[I.] $q(x) = q(x-\lambda)$ or equivalently $2\langle x,\lambda\rangle = q(\lambda)$, and
\item[II.] for all $\mu \in \Lambda$,  $2\langle x,\mu\rangle \leq q(\mu)$. 
\end{itemize}
By Lemma~\ref{lem:eulerian}, $\lambda$ is Eulerian; $\supp(\lambda)$ is thus an Eulerian oriented graph and can be decomposed into a disjoint union of  oriented cycles. In other words, there exist Eulerian $\lambda_1,\dots,\lambda_{k} \in \Lambda$, for some $k \in \mathbb N$, such that 
\begin{itemize}
\item[$(i)$] for $1\leq i\leq k$, $\supp(\lambda_i)$ is a circuit;
\item[$(ii)$] $\lambda = \lambda_1 + \dots+\lambda_{k}$; and
\item[$(iii)$] for all $1\leq i \neq j \leq k$, $\langle \lambda_i,\lambda_j\rangle = 0$.
\end{itemize}
In addition, for $\lambda_i$ one may choose any oriented circuit $C$ of $\supp(\lambda)$ and define $\lambda_1 = x^C$. 

\noindent (To see this, note that one way to obtain the decomposition is to proceed greedily, and in each step, choose an arbitrary circuit in the remaining oriented graph, delete it and proceed.)

Applying Property II above to each $\lambda_i$  and using $(ii)$ and $(iii)$ above, one sees that 
\begin{align*}\label{eq:plane}
2\langle x,\lambda_i\rangle = q(\lambda_i),\qquad \textrm{for all}\,\, i \in \{1,\dots,k\}.   
\end{align*}
Since there exists a decomposition such that $\lambda_1 = x^C$ for an arbitrary circuit of $\supp(\lambda)$, we infer that 
\begin{align*}
\hspace{2.5cm} \textit{For every circuit} \,\,\, C \subset \supp(\lambda),\qquad 2\langle x,x^C\rangle = q(x^C).\hspace{3.6cm} (*)
\end{align*}

\hspace{.3cm}

\noindent  Let $C_1,\dots,C_{g_\lambda}$ be a basis for the cycle space of $\supp(\mu)$ and let $P_\lambda$ be the affine plane defined by the $g_\lambda$ equations $(*)$ for $x^{C_i}$, for $i=1,\dots,g_\lambda$. Here $g_{\lambda}$ is the genus of $\supp(\lambda)$.
 Since the elements $x^{C_i}$ are independent, and the inner product $\langle\,,\rangle$ is non-degenerate,  $P_\lambda$ has codimension $g_\lambda$ and it contains $V_O\cap V_\lambda$. The lemma follows. 
\end{proof}

\begin{proposition}\label{prop:decomp}
Let $\Eu$ be the set of all Eulerian elements of $\Lambda$ of genus $1$.  $\Eu$ forms a system of generates of $\Lambda$.   The one-skeleton of the Delaunay dual of $\Vor_q$ is isomorphic to $\Cay(\Lambda,\Eu)$. 
\end{proposition}
\begin{proof}
Let $\mu \in \Lambda$. We show the existence of  $\lambda_1,\dots,\lambda_k$ such that each $\lambda_i$ is of the form $x^{C_i}$ for a circuit $C_i$ and in addition $\mu = \sum_i \lambda_i$ (and in addition $C_i \subset \supp(\lambda)$). This clearly proves the proposition.  

\noindent Let $D$ be the Eulerian oriented graph obtained by replacing each arc $e$ in $\supp(\mu)$ with $\lambda_e$ different parallel arcs. Since $D$ is Eulerian, it admits a decomposition into $k$ circuits, for some $k\in \mathbb N$. Looking at the circuits in the original oriented graph $\supp(\mu)$, one obtains $k$ circuits $C_1,\dots,C_k \subset \supp(\mu)$, and it is clear that $\mu = \sum_1^k \lambda_i$ for $\lambda_i =x^{C_i}$. 
\end{proof}

\paragraph*{Hyperplane Arrangement Defined by Circuits} Let $\Eu$ be the set of all the Eulerian elements of $\Lambda$ of genus one,  and let $\lambda$ be an element of $\Xi_1$. Let $\H_\lambda$ be the hyperplane in $\H$ defined by the equation 
$$\H_\lambda:=\bigl\{\,x\in \H\,|\, 2\langle x,x^\lambda  \rangle = q(\lambda)\,\bigr\}.$$
More generally, for an arbitrary lattice point $\mu \in \Lambda$ and an Eulerian  $\lambda \in \Lambda$ of genus one, let $\H_{\mu,\lambda}$ be the affine hyperplane $\mu+\H_\lambda$. By the results of the previous section, the Voronoi diagram $\Vor_q$ is the hyperplane arrangement of all the hyperplanes $\H_{\mu,\lambda}$ for $\mu \in \Lambda$ and $\lambda \in \Xi_1$. In addition, the Voronoi cell $V_O$ (and more generally $V_\mu$) is the cell containing $O$ (resp. $\mu$) in the hyperplane arrangement of the hyperplanes $\H_\lambda$ (resp. $\H_{\mu,\lambda}$) for $\lambda \in \Eu$.

\subsection{Definition of the map $\phi$ from $\FP$ To $\SC$} Each face $F$ in $\FP$ is obtained as an intersection of some of the hyperplanes $\H_\lambda$ for $\lambda \in \Eu$. 

\begin{lemma}\label{lem:cle}
Let $F\in \FP$ be a face of $V_O$, and let $\lambda$ and $\mu$ be two different elements of $\Eu$ such that $F \subset \H_\lambda \cap \H_\mu$. Then $\supp(\lambda)$ and $\supp(\mu)$ are consistent in their orientations: i.e., there is no oriented edge $e \in \mathbb E$ such that $e \in \supp(\lambda)$ and $\bar e \in \supp(\mu)$.
\end{lemma} 

\begin{proof}
Let $C_\lambda = \supp(\lambda)$ and $C_\mu = \supp(\mu)$, and for the sake of a contradiction, suppose there exists an oriented edge $e\in \mathbb E$ with $e \in C_\lambda$ and $\bar e \in C_\mu$.  We show that for every point $x$  in $V_O \cap \H_\lambda \cap \H_\mu$, there exists an element $\gamma \in \Eu$ which is strictly closer to $x$. This will show that the intersection $V_O \cap \H_\lambda \cap \H_\mu$ is empty, which will be a contradiction, and thus, the lemma follows. 

Let $x \in V_O \cap \H_\lambda \cap H_\mu$. We have 
\begin{equation*}
2\langle x,\lambda  \rangle = q(\lambda)=\ell(C_\lambda) \qquad \textrm{and} \qquad 2\langle x,\mu  \rangle = q(\mu) = \ell(C_\mu).
\end{equation*}
Thus, 
\begin{equation}\label{eq1}
2\langle x, \lambda + \mu  \rangle = \ell(C_\lambda) + \ell(C_\mu).
\end{equation}
Let $D = \supp (\lambda+\mu)$. Observe that neither $e$ nor $\bar e$ belong to $D$ and so the sum $\sum_{e\in D} \lambda_e+\mu_e$ is strictly less than $\ell(C_\lambda) + \ell(C_\mu)$. By (the proof of) Proposition~\ref{prop:decomp}, $\lambda+\mu$ can be written as a sum $\gamma_1+\dots+\gamma_k$ such that $\gamma_i \in \Eu$ and $\supp(\gamma_i) \subset D$. This shows that 
\begin{equation}\label{eq2}
\sum_i \ell(\supp(\gamma_i)) = \sum_{e\in D} \lambda_e+\mu_e < \ell(C_\lambda) + \ell(C_\mu).
\end{equation}
If for all $i$, $2\langle x,\gamma_i \rangle \leq q(\gamma_i) = \ell(\supp(\gamma_i))$, then by summing up all these inequalities and applying Inequality~\ref{eq2}, we should have
\begin{equation}
2\langle x,\lambda+\mu\rangle = \sum_i  2\langle x,\gamma_i \rangle  \leq \sum_i \ell(\supp(\gamma_i)) < \ell(C_\lambda) +\ell(C_\mu), 
\end{equation}
which is impossible by Equation~\ref{eq1}. Thus, there exists an $i$ such that $2\langle x,\gamma_i \rangle > q(\gamma_i)$. And clearly, in this case $q(x) > q(x-\gamma_i)$, contradicting the assumption that $x \in V_O$. 
\end{proof}

\paragraph*{Definition of the injective map $\phi$ from $\FP$ to $\SC$. }Applying Lemma~\ref{lem:cle}, we now construct a map $\phi:\FP \rightarrow \SC$ as follows. Let $F \in \FP$, and consider the family $\mathcal U_F = \{\lambda\in \Eu\,|\, F \subset \H_\lambda\}$. Define $\phi(F) = \bigcup_{\lambda \in \mathcal U_F} \supp(\lambda)$. By Lemma~\ref{lem:cle}, two different $\lambda$ and $\mu$ are consistent in their orientations, so the union in the definition of $\phi(F)$ defines a well-defined orientation of a subgraph of $G$. In addition, it is quite straightforward to check that $\phi(F)$ is a strongly connected orientation of that subgraph, i.e., $\phi(F)$ lies in $\SC$. 

\begin{lemma}\label{lem:cle2}
Let $F$ be a face in $\FP$, and $\mathcal U_F$ and $\phi(F)$ be defined as above. For any $\lambda\in \Eu$ with $\supp(\lambda) \subset \phi(F)$, we have $\lambda \in \mathcal U_F$, i.e., $F \subset \H_\lambda$. More precisely, $\lambda \in \mathcal U_F$ if and only if $\supp(\lambda) \subset \phi(F)$. 
\end{lemma}
\begin{proof}Define the flow $\mu = \sum_{\gamma \in \mathcal U_F} \gamma$. By applying Lemma~\ref{lem:cle} and the definition of $\mathcal U_F$, we have
\begin{equation}\label{eq3}
\forall x \in F \qquad  2\langle x,\mu\rangle = \sum_{\gamma \in \mathcal U_F} \ell(\supp(\gamma)),
\end{equation}
and since $F \subset V_O$, we have
\begin{equation}\label{eq4}
\forall x\in F, \gamma \in \Eu, \qquad 2\langle x,\gamma\rangle \leq q(\gamma)=\ell(\supp(\gamma)).
\end{equation}
As the prove of Proposition~\ref{prop:decomp} and Lemma~\ref{lem:codim} shows, since $\supp(\lambda) \subset \supp(\mu)  =\phi(F)$, there exists another decomposition of $\mu=\sum_{i=}^k \gamma'_i$ with $\gamma_1'=\lambda$ and such that $\supp(\gamma_i')\subset \supp(\mu)$. This shows that $\sum_{i=1}^k \ell(\supp(\gamma'_i)) = \sum_{\gamma \in \mathcal U_F} \ell(\supp(\gamma))$.   Equation~\ref{eq3} and Inequalities~\ref{eq4} for $\gamma'_i$ then imply that for all $i \in \{1,\dots,k\}$, $2\langle x,\gamma'_i\rangle = q(\gamma'_i)$, i.e., for all $i$, $\gamma'_i \in \mathcal U_F$, and in particular $\lambda =\gamma'_1 \in \mathcal U_F$.  The second part of the lemma follows from the first part and the definition of $\mathcal U_F$. 
\end{proof}
\begin{lemma}
The map $\phi$ is injective and order-preserving. 
\end{lemma}
\begin{proof}
Injectivity follows easily from the previous lemma: if for two faces $F_1$ and $F_2$ we have $\phi(F_1) = \phi(F_2)$, then $\mathcal U_{F_1} = \mathcal U_{F_2}$. Since $V_O$ is the cell containing the origin in the arrangement of the hyperplanes $\H_\lambda$ for $\lambda \in \Eu$, this shows that $F_1 = F_2$. 

\noindent To show that the map $\phi$ is order preserving, note that if $F_1 \subseteq F_2$ for two faces in $\FP$, then $\mathcal U_{F_2} \subseteq \mathcal U_{F_1}$, thus, $\phi(F_2) =  \bigcup_{\lambda \in \mathcal U_{F_2}} \supp(\lambda) \subseteq \bigcup_{\lambda \in \mathcal U_{F_1}}  \supp(\lambda)=\phi(F_1)$ and so by the definition of the partial order in $\SC$, we have $\phi(F_1) \preceq \phi(F_2)$.  
\end{proof}
\subsection{Explicit description of the faces in $\FP$}
We now show that the map $\phi$ defined in the previous section is surjective. 

\noindent Let $D$ be a strongly connected orientation of a subgraph  $H$ of $G$. Let $\H_D$ be the affine plane defined as the intersection of all the affine hyperplanes $\H_{\lambda}$ for $\lambda \in \Eu$ and $\supp(\lambda) \subseteq D$. From the analysis we did before, we know that $\codim (\H_D) = g_D$, the genus of $H$. We show that there is a face $F$ of $\FP$ of codimension $g_D$ which is included in $\H_D$ and $\phi(\FP) = D$. This proves the surjectivity of $\phi$ and Theorem~\ref{thm:main} follows. 

In what follows, by the abuse of the notation, by a strongly connected orientation of the whole graph we simply mean a strongly connected orientation of the graph minus the set of bridges. (Thus, if $G$ does not have any bridge, i.e., is 2-edge connected, the orientation is indeed defined on the whole graph.)
\begin{lemma}
If $D$ is a strongly connected orientation of the whole graph $G$, then $\H_D$  consists of precisely a vertex of $V_O$, denoted by $v^D$.  All the vertices of $V_O$ are of this form. 
\end{lemma}
\begin{proof}
We already know that $\H_D$ is of codimension $g_D = g$. In other words, $\H_D = \{v^D\}$ for a point $v^D \in \H$. We now show that $v^D$ is a vertex of $V_O$.  To this end, we have to show that for all $\lambda \in \Eu$, $q(v^D) \leq q^(v^D -\lambda)$, or equivalently $2\langle v^D,\lambda \rangle \leq q(\lambda) = \ell(\supp(\lambda))$. Let $\lambda \in \Eu$. Since $D$ is a strongly connected orientation of $D$, the circuits of $D$ form a basis for $\Lambda$. There exists thus a set of coefficients $a_\mu$ for $\mu\in \Eu$ with $\supp(\mu) \subset D$ such that 
$$\lambda = \sum_{\mu \in \Eu:\, \supp(\mu) \subset D} a_\mu \, \mu.$$
We have 
\begin{align*}
2\langle v^D,\lambda \rangle&=  \sum_{\mu \in \Eu: \, \supp(\mu) \subset D} 2a_\mu\langle v^D, \mu \rangle\\
&=  \sum_{\mu \in \Eu:\, \supp(\mu) \subset D} a_\mu \ell(\supp(\mu))\\
&= \sum_{e\in D} \Bigl( \sum_{\mu \in \Eu: \, e\in \supp(\mu) \& \supp(\mu) \subset D}  a_\mu\Bigr)= \sum_{e\in D} \lambda_e \\
&\leq \sum_{e\in D} |\lambda_e| = \ell(\supp(\lambda)) = q(\lambda). 
\end{align*}
In addition, the last inequality is strict if and only if $\supp(\lambda) \not\subset D$. 
We infer that $v^D \in V_O$ and it is a vertex of $V_O$. 

To show that a vertex $v$ of $V_O$ is of this form, consider the image of the face $\{v\}$  by  $\phi$. This is a strongly connected orientation of a subgraph $H$ of $G$, and since the intersection $\cap_{\lambda \in \mathcal U_{\{v\}} } H_\lambda $ is of codimension $g$, the genus of $H$ should be equal to that of $G$, i.e., $H =G$ and thus, $v=v^{\phi(\{v\})}$ and the lemma follows. 
\end{proof}

Consider now the general case, $D$ being a strongly connected orientation of a subgraph $H$ of $D$. 
Let $\mathcal S_D$ be the family of all the extensions of the orientation of $H$ to a strongly connected orientation of the whole graph $G$, i.e., the family of all the strongly connected orientations $D'$ of $G$ such that $D \subseteq D'$. 
For each element $D' \in \mathcal S_D$, let $v^{D'}$ the vertex of $V_O$ given by the previous lemma. By the convexity of $V_O$ and $\H_D$, we have $$\conv_{D' \in \mathcal S_D}\langle \, v^{D'}\rangle \subset V_O \cap \H_D.$$ 
\begin{lemma}
The convex-hull of $v^{D'}$ for $D' \in \mathcal S_D$ defines a face $F$ of  $V_O$ and we have $\phi(F) = D$. Thus, the map $\phi$ is surjective. 
\end{lemma}
\begin{proof}
Let $F= \conv_{D' \in \mathcal S_D}\langle \, v^{D'}\rangle$. We first show that $F$ is a face of $V_O$, the second statement $\phi(F)=D$ easily follows. 
By the previous lemma, any vertex  $v$ of $V_O$ which lies in $\H_D$ is of the form $v^{D'}$ for a strongly connected orientation $D'$ of $G$. Since $\mathcal U_D \subset \mathcal U_{D'}$, and $D$ and $D'$ are strongly connected orientation, we infer that $D \subseteq D'$, i.e., $D'\in \mathcal S_D$, and so all the vertices of $V_O$ which lie in $\H_D$ are of the form $v^{D'}$ for $D'\in \mathcal S_D$. It follows that $F$ is a face of $V_O$.
To show that $\phi(F) =D$, note that for any $\lambda \in \Eu$, by the proof of the previous lemma, we have if $\supp(\lambda) \not\subset D$, then for every $D' \in \mathcal S_D$, $2\langle v^{D'},\lambda \rangle < \ell(\supp(\lambda))$, i.e., $\lambda \notin \mathcal U_F$.  And clearly if $\supp(\lambda) \subset D$, then for every $D' \in \mathcal S_D$, $2\langle v^{D'},\lambda \rangle = \ell(\supp(\lambda))$, and since $F = \conv_{D' \in \mathcal S_D}\langle \, v^{D'}\rangle$, we have $2\langle x,\lambda \rangle = \ell(\supp(\lambda))$ for every $x\in F$, i.e., $\lambda \in \mathcal U_F$. We conclude that $\phi(F) =D$, and the lemma follows. 
\end{proof}
The proof of Theorem~\ref{thm:main} is now complete. 
Note that, as the proof shows, the codimension of a face $F\in \FP$ is equal to the genus of $\phi(F)$.

\section{Quotient posets $\overline \FP$ and $\overline \SC$}   \label{sec:secondpart}
The Voronoi diagram $\Vor_q$ is naturally equipped with the action of $\Lambda$ obtained by translation. (More precisely this is defined as follows: for an element $\lambda \in \Lambda$, and a face $F $ of a Voronoi cell $V_\mu$, $\mu.F := \mu+F$ is the corresponding face of the Voronoi cell $V_{\lambda+\mu}$ under the natural isomorphism obtained by translation). The poset $\overline \FP$ is defined as the quotient of $\Vor_q$ under this action. Another way to define $\overline\FP$ is to start from $\FP$ and identify those faces of $\FP$ which differ by a translation by an element of $\Lambda$.

There is also a quotient poset $\overline \SC$ associated to $\SC$ obtained roughly by identifying those strongly connected orientations  of subgraphs of $G$ which differ by reversing the orientation of an Eulerian orientation of an Eulerian subgraph of $G$.  More precisely, $\overline \SC$ is defined as follows. 

 Let $D \in \SC$ be a strongly connected orientation of a subgraph $H$ of $G$.  The outdegree function $d^+_D:\ V \rightarrow \mathbb N$  is define by sending a vertex $v$ of $G$ to the number of oriented edges emanating from $v$ in the orientated graph $D$, i.e., the number of arcs $e$ in $D$ whose tail is $v$. The equivalence relation $\sim$ is defined on $\SC$ as follows. For two strongly connected orientations $D_1$ of a subgraph $H_1$ and $D_2$ of a subgraph $H_2$ of $G$, we set $D_1\sim D_2$ if an only if 
 \begin{itemize}
 \item $H_1 =H_2$, and
 \item  the two corresponding out-degree functions are equal, i.e., $d^+_{D_1} = d^+_{D_2}$. 
 \end{itemize}
 \begin{lemma}
Two strongly connected orientations $D_1$ and $D_2$ of a subgraph $H$ of $G$ are equivalent if and only there exists an Eulerian subgraph $H_0$ of $H$ and an Eulerian orientation $D_0$ of $H_0$ such that $D_0 \subset D_1$ and $D_2$ is the orientation obtained by replacing each edge $e $ in $D_0$ by $\bar e$, i.e., $\mathbb E(D_2) = \Bigl(\mathbb E(D_1) \setminus \mathbb E(D_0)\Bigr) \cup \bigl\{\bar e\,|\, e \in \mathbb E(D_0)\bigr\}$. 
 \end{lemma}
\begin{proof} The proof is quite easy: consider the set of all oriented edges $e \in \mathbb E(D_1)$ such that $\bar e \in D_2$. This defines an Eulerian oriented graph $D_0$ and lemma easily follows. 
 \end{proof}

 The poset $\overline \SC$ is the quotient of $\SC$ obtained by identifying equivalent elements of $\SC$, in other words $\overline SC = \SC /\sim$. Note that this is indeed a poset: if $D_1 \preceq D'_1$, and $D'_1 \sim D'_2$, then by the definition of the partial order $\preceq$ and the lemma above, there exists an Eulerian orientation $D'_0$ of a subgraph of $G$ such that $D'_0 \subseteq D'_1 \subseteq D_1$ and $D'_1$ is obtained from $D_1$ by  reversing the orientation of the oriented edges of $D'_0$. Now reversing the orientation of the oriented edges of $D'_0$ in $D_1$, one obtained an element $D_2 \in \SC$ such that $D_2 \sim D_1$ and $D_2 \preceq D'_2$. So the quotient partial order $\preceq$ is well-defined.  

\begin{theorem}\label{thm:mainq}
The two posets $\overline \SC$ and $\overline \FP$ are isomorphic. 
\end{theorem}
\begin{proof}
The proof follows from the proof of our Theorem~\ref{thm:main} and Lemma~\ref{lem:eulerian}. Namely, consider the bijective map $\phi$ defined in the previous section. Let $F_1$ and $F_2$ be two faces in $\FP$ which are identified in $\FP$, i.e., there exists an element $\lambda \in \Lambda$ such that $F_2 = F_1 +\lambda$. Since $F_2$ is also a face of $V_\lambda$, we infer that $V_\lambda \cap V_O \neq \emptyset$, and by Lemma~\ref{lem:eulerian}, $\lambda$ is Eulerian. Let $D_0 = \supp(\lambda)$, note that $D$ is Eulerian.  We show that $D_0 \subset \phi(F_2)$ and $\phi(F_1)$ is obtained by reversing the orientation of $D_0$ in $\phi(F_2)$, from this we obtain a well-defined quotient map $\overline \phi$, which becomes automatically bijective, and the theorem follows. 

By the definition of $\phi$, we have to show that for every $\mu \in \Eu$ with $\supp(\mu)$ an oriented cycle in $D_0$, we have $\mu \in \mathcal U_{F_2}$. Let $\mu \in \Eu$ be such that  $\supp(\mu) \subset D_0$.   Since $F_1$ is a face of $V_O$, for all $x \in F_1$, we have $-2\langle x, \mu \rangle =2\langle x, -\mu \rangle \leq q(-\mu ) =q(\mu)= \ell(\supp( \mu) )$. It follows that 
\begin{align*}
2\langle x+\lambda, \mu \rangle &= 2\langle x, \mu \rangle + 2 \ell(\supp(\mu)) \qquad \textrm{since $\supp(\mu) \subseteq \supp(\lambda)$}\\ 
&\geq \ell(\supp(\mu)) =q(\mu). 
\end{align*}
Since $x+\lambda \in F_2$, and $F_2$ is a face of $V_O$, the above inequality is indeed an equality for all $x\in F_1$, or equivalently for all $x+\lambda \in F_2$, i.e., for all $y \in F_2$, $2\langle y, \mu \rangle = q(\supp(\mu))$. Thus, $\mu \in \mathcal U_{F_2}$. We infer that $D_0=\supp(\lambda) \subset \phi(F_2)$. 

Let $D_1 = \phi(F_1)$, $D_2 = \phi(F_2)$ and $D'_1$ the element of $\SC$ obtained from $D_2$ by reversing the orientation of every oriented edge in $D_0 \subset D_2$. We aim to show that $D'_1 =D_1$.  

In the proof of Theorem~\ref{thm:main}, we show that $F_2$ (resp. $F_1$) is the convex hull of all the vertices $v^D$ of the Voronoi cell $V_0$ for $D \in \mathcal S_{D_2}$ (resp. $D \in \mathcal S_{D_1}$), i.e., for $D$ being a strongly connected orientation of $G$ which contains  $D_2$ (resp. $D_1$) as an oriented subgraph.   On the other hand, since $F_1 = F_2+\lambda$, we have $\conv_{D \in \mathcal S_{D_2}}(v^D+\lambda) = F_1$. We now show that the set $\{v^D+\lambda\}$ coincides with $\mathcal S_{D'_1}$, i.e., with the set of all  the strongly connected orientations of $G$ which contains $D'_1$ as an oriented subgraph. From this, and the proof of Theorem~\ref{thm:main},  it easily follows that $D'_1=D_1$. 

Recall that $v^D$ has the property that for every $\mu \in \Eu$ with $\supp(\mu) \subset D$, 
$$2\langle v^D,\mu\rangle =\ell(\supp(\mu)).$$ 
Since $D$ is in $\SC$, the set of $\mu$'s with $\supp(\mu) \subset D$ generate $\Lambda$. And since the above equations are all linear, we easily have for all $\mu \in \Eu$ (not only $\mu$ with $\supp(\mu) \subset D$),
$$2\langle v^D,\mu\rangle =\ell^+(\supp(\mu)) - \ell^-(\supp(\mu)),$$
where $\ell^+(\supp(\mu))$ (resp. $\ell^-(\supp(\mu))$) is the number of edges $e$ of $\supp(\mu)$ with $e\in \mathbb E(D)$ (resp. with $\bar e \in \mathbb E(D)$). From this, we have for all $\mu \in \Eu$, 
\begin{align}\label{eq:quotient}
2\langle v^D+\lambda, \mu \rangle = \ell^+(\supp(\mu)) - \ell^-(\supp(\mu)) + 2\langle \lambda, \mu \rangle
\end{align} 
We claim that $v^D+\lambda$ is equal to $v^{D'}$ for the orientation $D'$ obtained by reversing the orientation of $D_0$ in $D$ (note that $D_0 \subset \phi(F_2) =D_2 \subset D$). For this, we have to show that for all $\mu \in \Eu$ with $\supp(\mu) \subset D'$,  we have $2\langle v^D+\lambda, \mu \rangle = \ell(\supp(\mu))$. But this easily follows from the Equation~\ref{eq:quotient} above. Indeed, for such a $\mu$ we have
\begin{align*}
2\langle v^D+\lambda, \mu \rangle &= \ell^+(\supp(\mu)) - \ell^-(\supp(\mu)) + 2\langle \lambda, \mu \rangle\\
&= \ell^+(\supp(\mu)) - \ell^-(\supp(\mu)) + 2 \ell^-(\supp(\mu)) \\
&= \ell(\supp(\mu)). 
\end{align*} 
The proof of the theorem is now complete. 
\end{proof}

\section{Concluding remarks}
As we said, the results of the previous sections on the lattice of integer flows extend directly to the weighted case, i.e., to the case of graphs with an edge length function.

\subsection*{Covering number and well-(un)balanced orientations}
Let $\mathcal L$ be a lattice of full dimension in a real vector space $\mathcal E$ equipped with a positive quadratic form $q$.   The {\it covering number} of $\mathcal L$ denoted by $\cov(\mathcal L)$ is the smallest real number $r$ such that the closed balls of radius $r$ centered at the lattice points cover $\mathcal E$, i.e., 
such that for every point $x\in \mathcal E$, there is an element $m\in \mathcal L$ such that $q(x-m) \leq r$.

 By the results of the previous section, it is possible to give an explicit description of the covering number of $\Lambda$ in terms of certain orientations of the graph $G$ (that we will call balanced) as follows.  

 Recall that all the vertices of the Voronoi cell of the origin, $V_O$, are of the form $v^D$ for a strongly connected orientation $D$ of the whole graph $G$. In addition, $v^D$ is the flow uniquely determined by the equations 
 
 $$ 2\langle v^D , x^C\rangle = \ell(C),$$
 for all circuits $C \subset D.$
 \begin{proposition}
 The covering number of $\Lambda$, $\cov(\Lambda)$, is the maximum of $q(v^D)$ for $D$ a strongly connected orientation of the whole graph $G$.
 \end{proposition}
 Thus, it will be enough to find a strongly connected orientation $D$ of $G$ such that $q(v^D)$ is maximized.
 
 Let $\Delta$ be the Laplacian operator. We recall that for a function $f : V \rightarrow \mathbb R$, $\Delta(f)$ is the function on  $V$ taking on $y \in V$ the value $\Delta(f)(y) := \sum_{z: \,\{z,y\} \in E(G)} \bigl (f(y)-f(z)\bigr)$.
 
 Let $\mathcal S$ be the family of all the strongly connected orientations of the whole graph. (Recall that if $G$ contains bridges, by our convention as before, we mean all the strongly connected orientation of $G$ minus the set of bridges of $G$.) 
 \noindent Let $D$ be a fixed element of $\mathcal S$, and for a vertex $y \in V$, let $\din_y$ and $\dout_y$ be the in- and the out-degree of $y$ in $D$ respectively.
  
\noindent  Define the function $\mu : \E(D) \rightarrow \mathbb R$ on the set of arcs of $D$ as follows: for all $e \in \E(D)$, $\mu_e = 2v^D_e -1$.  Then we have
$$\textrm{For all circuit }\,\, C \subset D,\,\,\,\, \langle \mu, x^C\rangle =0. $$
  (In other words, $\mu$ is a tension on $D$.)
  
\noindent  It follows that there exists a function $f=f_D: V \rightarrow \mathbb R$ such that for all arc $e = (y,z) \in \E(D)$, $\mu_e = f_D(z) -f_D(y)$.  (In other words, there exists a potential function for $\mu$.)

In addition, since $v^D$ is a flow, the function $f$ should satisfy the following set of equations:

$$\forall y\in V, \,\, \sum_{z: \, (z,y) \in \E(D)} \bigl(f(y)-f(z)+1\bigr) \,\,\,\, = \sum_{z: \, (y,z) \in \E(D)} \bigl(f(z) -f(y)+1\bigr).$$
Or equivalenty,
$$\forall y \in V, \,\, \Delta(f)(y) = \dout_y-\din_y.$$
In other words, $f$ is the solution to the Laplacian equation $\Delta(f) = \dout -\din.$ We now show how to describe $q(v^D)$ in terms of $f$. First we need the following lemma. 

 \begin{lemma}
We have $2\,q(v^D) = |v^D|_{\ell_1} = \sum_{e \in \E(D)} v^D_e$.
 \end{lemma}
 \begin{proof}
We already know that $\supp(v^D) =D$. Write $v^D$ as a positive combination of the flows $x^C$ for $C \subset D$,
$v^D= \sum_{C \subset D} \alpha_C\, x^C.$
We have 
\begin{align*}
2\, q(v^D) &= \sum_{C\subset D}\, 2\alpha_C \langle v^D, x^C\rangle \\
&=\sum_C \alpha_C \ell(C) = \sum_{e \in \E(D)} \sum_{C: \,e\in C} \alpha_C \\
& = \sum_{e \in \E(D)} v^D_e = |v^D|_{\ell_1}.
\end{align*}
 \end{proof}

Since for $e = (y,z) \in \E(D)$, we have $v^D_e = f(z) -f(y)+1$, it follows that 
\begin{align*}
2q(v^D) &= \sum_{(y,z) \in \E(D)}\, \bigl(f(z) -f(y)+1\bigr) \\
& = |\E(D)| + \sum_{y \in V} (\din_y -\dout_y) f(y)\\
&=  |\E(D)| - \sum_{y\in V} f(y)\Delta(f)(y) = |\E(D)| - \langle f,\Delta(f)\rangle.
\end{align*}
 In addition, $\E(D)$ is the number of the non-bridge edges of $G$. Thus, we have
 \begin{definition-theorem}{\rm
A strongly connected orientation $D$ of $G$ for which the quantity $\langle f_D,\Delta(f_D)\rangle$ is  minimized (for $f_D$ satisfying the Laplacian equation  $\Delta(f_D) = \dout -\din$)  is called well-balanced.}

 Let $\epsilon$ be the number of the non-bridge edges of $G$. Let $D$ be  a well-balanced orientation of $G$. The covering number of $\Lambda$ is  given by 
 $$\cov(\Lambda) = \frac \epsilon 2  -  \frac 12 \langle f_D,\Delta(f_D)\rangle.$$
 \noindent In particular if $G$ is Eulerian, then $\cov(\Lambda) = \frac {|E(G)|}2$. 
 \end{definition-theorem}

\begin{proof}
It only remains to prove the last claim: if $G$ is Eulerian, then $\epsilon = |E(G)|$, and there exists an orientation $D$ of $G$ such that $\din_y=\dout_y$ for all $y\in V$. Clearly $D$ is strongly connected and we have $\Delta(f_D) = 0$, and so $\langle f_D,\Delta(f_D)\rangle=0$, i.e., $D$ is well-balanced! The claim follows. 
\end{proof}

Thus, we are left with the following optimization problem.
\begin{center}
Find a strongly connected orientation $D$ of $G$ such that for $f$ satisfying  $\Delta(f) = \dout -\din,$\\
the quantity $\sum_{y\in V} f(y) (\dout_y - \din_y)$ is minimized. 
\end{center}

 We do not know the complexity of the above problem in general, however, we suspect it to be $\mathscr{NP}-$hard.

\begin{remark}\rm Suppose for simplicity that $G$ is two-edge connected. Let $c(y) := \din_y - \dout_y$ for all $y \in V$. Then the following properties are satisfied
\begin{itemize}
\item[$(i)$] $\sum_y\,\,c(y) = 0\,;$
\item[$(ii)$] for all $y\in Y$, \,\, $c(y) \equiv \deg(y)\,\, (\,\mathrm{mod}\,\,2)\,;$ and
\item[$(iii)$] for all non-empty $X \subsetneq V$, $c(X) = \sum_{y \in X} c(y) < \delta(X),$ where $\delta(X)$ is the number of edges with one end-point in $X$ and one end-point  in $V\setminus X$. 
\end{itemize}
It is not hard to show that the converse is also true: a function $c:V \rightarrow \mathbb Z$ is of the form $\din-\dout$ for a strongly connected orientation $D$ of $G$ if and only if it satisfies the properties $(i), (ii),$ and $(iii)$ above (and in this case, it is possible to find the orientation in polynomial time). Thus, while calculating the covering number, we are left with the optimization problem of finding the minimum of $\sum_{y \in V} f(y)c(y)$ over all the functions $c$ with the above three set of properties and for $f$ satisfying the Laplacian equation $\Delta(f) = c$.
\end{remark}

\end{document}